\documentclass[12pt,a4paper]{article}
\pdfoutput=1 
\usepackage[T1]{fontenc}
\usepackage[utf8]{inputenc}
\usepackage[english]{babel}
\usepackage{authblk}
\usepackage{amsmath,amssymb,amsthm}
\usepackage{graphicx}
\usepackage{color}
\usepackage{hyperref}
\usepackage{todonotes}
\usepackage{amsfonts}
\usepackage{mathrsfs}
\usepackage{url}
\newtheorem{theorem}{Theorem}
\newtheorem{lemma}{Lemma}
\newtheorem{remark}{Remark}

\allowdisplaybreaks

\begin{document}
	
\begin{center}
	\textbf{{\Large Branching Random Walks and their Applications for Epidemic Modelling}}
\end{center}

\begin{center}
	\textit{{\large Elizaveta Ermakova, Polina Makhmutova,  Elena Yarovaya}}
\end{center}

\begin{abstract}
Branching processes are widely used to model the  viral epidemic evolution. For more adequate investigation of viral epidemic modelling, we suggest to apply branching processes with transport of particles usually called branching random walks (BRWs). This allows to investigate not only the number of particles (infected individuals), but also their spatial spread. We consider two models of continuous-time BRWs on a multidimensional lattice in which the transport of infected individuals is described by a symmetric random walk on a multidimensional lattice whereas the processes of birth and death of infected individuals are represented by a continuous-time Bienayme-Galton-Watson processes at the lattice points (branching sources). A special attention is paid to the properties of branching random walks with one branching source on the lattice and finitely or infinitely many initial particles. We  show that there exists a kind of duality between the branching random walk  with a finite number of initial particles and the branching random walk with an infinite number of initial particles, which is associated with the possibility of their twofold description. The fact of duality  is useful from the biological point of view. Each of the models can be considered taking into account the vaccination process. We suppose the vaccination to be a proportion of immune individuals in the population, who are resistant to disease. For simplicity, in all our BRW models, we assume that the vaccination process does not depend on time, what allows to investigate spatial properties of viral evolution.
\end{abstract}

\section{Introduction}

Branching processes are widely used in investigating the viral epidemic evolution~\cite{GMS-B:10,BGMS-B:10,FF:14,W:14}. In the work, we suggest to apply for viral epidemic modelling an important class of stochastic processes, the branching processes with the transport of particles which is commonly called the branching walks or branching random walks (BRWs). This allows us to investigate not only behavior of the particle population, but also its spatial distribution. For this, we consider a few models of continuous-time BRWs on the multidimensional lattice $\mathbb{Z}^d$, $d\ge 1$. In these models, the transport of particles (infected individuals or virus particles) is described by a symmetric random walk on $\mathbb{Z}^d$, see, e.g.,~\cite{YarBRW:r}. Processes of birth and death of particles are represented by a continuous-time Bienayme-Galton-Watson processes at the lattice points, called branching sources, see, e.g.,~\cite{YarBRW:r}.
A special attention will be paid to the properties of two BRW models with a single branching source on $\mathbb{Z}^d$, one of these models with a finite and another with an infinite number of initial particles. Some results about connection between these  models are presented.

We will show that there exists a kind of duality between the BRW  with a finite number of initial particles and the BRW with an infinite number of initial particles which is associated with the possibility of their twofold description, see Sections~\ref{BRWI} and~\ref{BRWIF}, and which is useful from biological point of view.
The transition from the model with a finite number of initial particles to another model, with an infinite number of particles, gives an opportunity to describe the behavior of some infected populations in a more natural way. For example, in connection with this it is worth to mention the phenomenon called virus persistence~\cite{RG:17,BAP:96} according to which some viruses can get into a latent state, weakly interfering in cell processes and become active only under specific conditions. Such a state, called latency, is typical for herpes viruses, in particular the varicella or Epstein-Barr virus~\cite{BAP:96}, which results in infectious mononucleosis, see, e.g.,~\cite{VC:04}. The latent stage is also  presented in the construction of propagation strategy of some bacteriophages~\cite{CMLH:11}. Here, until the infected cell remains in an unfriendly environment, virus does not kill it, inherits by filial cells and integrates into the cell’s genome. But when the infected bacteria enters a favorable environment, the infectious agent captures the control over cell’s processes and it begins to produce materials for building new viral particles. Thousands of them move out of the cell, break the membrane, leading the cell’s death.
With virus persistence (for example, papovaviruses) some oncological diseases are related, see, e.g.,~\cite{B:96}.

In the above mentioned BRW model with an infinite number of initial particles every particle at a point of the lattice at initial time could present a virus in the latent stage. Then the branching source is a place with a favorable environment for the disease progression. Getting into it, the virus enter the active stage from the latency, so the transmission and the infection of other members of the population become possible.

Both of the BRW models can be considered taking into account the vaccination process. Unlike to~\cite{GMS-B:10}, we suppose the vaccination to be presented by a reduced level of replication possibility of virus particles. For simplicity, in both BRW models, we assume that the vaccination process does not depend on time, in contrast, e.g., with~\cite{BGMS-B:10}, what allows to investigate spatial properties of the viral evolution.

The structure of the work is as follows. In this paper we consider only BRWs with a single branching source. In Section~\ref{BRWI}, a formal description of a BRW with an infinite number of initial particles is recalled and the concept of a particle subpopulation is introduced. Besides, in Section~\ref{BRWI} the main differential and integral equations for the generating functions and the moments of the subpopulation particle number, as well as for the moments of a particle  number of every subpopulation at an arbitrary lattice point, are obtained.  In Section~\ref{BRWIF} we establish the duality of the BRW with an infinite number of initial particles and the BRW  with one initial particle situated at an arbitrary lattice point. This result allows us to apply the theorems obtained earlier for a well-known model with one initial particle (see, e.g.,~\cite{YarBRW:r}), for studying the BRW with an infinite number of particles and a finite variance of jumps, see Theorem~\ref{T3}. Let us point out that the assumption about the finiteness of the variance of jumps is used only in Theorem~\ref{T3}. At last, in Section~\ref{BRWV} we generalize both BRW models for possible applications taking into account the vaccination process.

\section{BRW with infinitely many initial particles}\label{BRWI}

We consider a continuous-time branching random walk on $\mathbb{Z}^d$, $d\ge 1$, with an infinite number of initial particles. The process of birth and death of particles occurs only at a single lattice point called the branching source. Without loss of generality, we can assume that the source is at the origin. Let just one particle be at every point $x\in \mathbb{Z}^d$ at instant $t=0$. Informal description of the model is rather simple. Being outside of the source every particle can walk on $\mathbb{Z}^{d}$ until
reaching the source. At the source it spends an exponentially distributed time
    and then either jumps to a point $x'\neq 0$
or can give a random number of offsprings (we assume that the
particle itself is included in this number; in this case, when the number of offsprings equals zero we say that the
particle dies). The newborn particles behave independently and
    stochastically in the same way as a parent individual.

    Denote by $\eta_{t}(y)$ the number of particles at the point $y\in\mathbb{Z}^d$ at time $t$. Then due to the assumption that there is just one particle at every point $x\in\mathbb{Z}^d$ at instant $t=0$ we have that the function $\eta_{t}(y)$ satisfies the following initial condition: $\eta_{0}(y)=1$ for all $y\in\mathbb{Z}^{d}$.
In this model we also divide all the particles on the lattice into subpopulations, that is, we assume that each initial particle at a point $x\in \mathbb{Z}^{d}$ is a progenitor of its separate subpopulation $\eta_{x,t}$ on $\mathbb{Z}^{d}$ at time $t$. Denote by $\eta_{x, t}(y)$  the number of particles from each subpopulation at time $t$ at the point $y\in\mathbb{Z}^d$ with the initial position of the progenitor particle at the point $x$. It satisfies the initial condition $\eta_{x, 0}(y)=\delta_{x}(y)$. Thus, we have  for every $y\in\mathbb{Z}^{d}$ the equality $\eta_{t}(y)=\sum_{x\in \mathbb{Z}^d}\eta_{x, t}(y)$ and for very $x \in\mathbb{Z}^{d}$ the equality $\eta_{x,t}=\sum_{y\in \mathbb{Z}^d}\eta_{x, t}(y)$ (subpopulation size on $\mathbb{Z}^{d}$).



The random walk of a particle on $\mathbb{Z}^{d}$ is defined by an infinitesimal matrix of transition intensities
$A=\bigl(a(x,y)\bigr)_{x,y\in \mathbb{Z}^{d}}$,
satisfying
$\sum_{y \ne x}a(x,y)=|a(x,x)|<\infty$ for all~$x$, where $a(x,y)\ge 0$
for $x\ne y$ and $a(x,x)<0$. We assume that the intensities $a(x,y)$ are symmetrical and homogeneous in space: $a(x-y):=a(x,y)=a(y,x)=a(0,y-x)$.
Furthermore, we suppose,
that the random walk is irreducible, that is, for each
$z\in\mathbb{Z}^{d}$ there exist
$z_{1},\dots,z_{k}\in\mathbb{Z}^{d}$, such that $z=\sum_{i=1}^{k}z_{i}$ and
$a(z_{i})\ne0$ for $i=1,\dots,k$. It means  that  each point $ y\in\mathbb{Z}^d $ is reachable.

We assume that for any small time $h$ the particle jumps from the point $x$  to the point $y\ne x$
with probability
\begin{equation}\label{E1}
p(h,x,y)=a(x,y)h+o(h),
\end{equation}
or remains in place with probability
\begin{equation}\label{E2}
p(h,x,x)=1+a(x,x)h+o(h).
\end{equation}

This implies that $p(t,x,y)$, with $y\in\mathbb{Z}^{d}$ treating as a parameter, satisfies the system of backward Kolmogorov equations:
\begin{equation}\label{E:Kolm}
\partial_tp(t,x,y) = \sum_{x'}a(x-x')p(t,x',y),\quad
p(0,x,y)=\delta_y(x),
\end{equation}
where $\delta_y(\cdot)$ is the discrete Kronecker $\delta$-function on $\mathbb{Z}^d$.

By introducing the linear operator $\mathscr{A}$ on $l^{p}(\mathbb{Z}^{d})$ for every $p$, $1 \leq p \leq \infty$, as follows
\[
(\mathscr{A}u)(x):=\sum_{x'}a(x,x')u(x')\equiv \sum_{x'}a(x-x')u(x'), \quad u(\cdot)\in l^{p}(\mathbb{Z}^{d}),~x\in\mathbb{Z}^{d},
\]
we can rewrite the Kolmogorov's equations~\eqref{E:Kolm} in the operator form:
$$
\partial_tp(t,x,y) = (\mathscr{A}p(t,\cdot,y))(x),\quad
p(0,x,y)=\delta_y(x),
$$
where $y$ is treated as a parameter.

The branching mechanism is defined by a continuous-time Bienayme-Galton-Watson process (see, e.g.,~\cite{Sev}) which is determined by the following infinitesimal generating function
\begin{equation}\label{pf}
f(u):=\sum_{n=0}^{\infty}b_nu^n,\quad 0\le u\le 1,
\end{equation}
where $b_n\ge0$ for $n\neq 1$, $b_1<0$ and $\sum_nb_n=0$.


By $p(h, n)$ we denote the probability to produce $n\neq 1$ offspring particles by the progenitor particle situated at a point $x\in \mathbb{Z}^{d}$ in a small time $h$ under condition $\eta_{x, 0}(y)=\delta(y-x)$:
\begin{align}\label{pn}
p(h, n) &= b_nh + o(h)\quad \text{for}\quad  n\neq 1,\\
p(h, 1) &= 1 + b_1h + o(h).
\end{align}

We assume also that there exists the first derivative of the generating function $\beta_{1} := f'(1)< \infty $, that is, the first moment of the direct particle offsprings is finite. Denote in this case $\beta := \beta_{1}$. The assumption that all moments are finite, that is, $\beta_{r} := f^{(r)}(1)< \infty$ for all $r$, is used in the method-based proofs of the limit theorems on behavior of the numbers of particles in BRW (see, e.g.,~\cite{YarBRW:r}).

Now, let us combine the branching and walking mechanisms in order to get a branching random walk with the source at $x=0$. Suppose that the particles may be at the source as well as at an arbitrary lattice point. Then at time $t$ the particles in small time $h$ may jump to a point $y\neq x$ with the probability
\begin{equation}\label{E:phxy}
p(h, x, y) = a(x, y)h + o(h)
\end{equation}
or, as earlier, it will stay at the point $x$ with the probability
\begin{equation}\label{E:phxx}
p(h,x,x)=1 +a(x,x)h + o(h)= 1 +a(0,0)h + o(h),
\end{equation}
and in this last case (staying at the source $x=0$) the particle can produce also $n\neq 1$ offsprings or die (case with $n=0$) with the probability
\begin{equation}\label{E:phn}
p_{*}(h, n) = p(h,0,0)\cdot p(h,n)=b_nh + o(h),
\end{equation}
or remain unchanged with the probability
\begin{equation}\label{E:ph1}
p_*(h, 1)=p(h,0,0)\cdot p(h,1) = 1+ a(0,0)h +b_{1}h + o(h).
\end{equation}

As a result, each particle stays at the source for a random time, exponentially distributed with the parameter $-(a(0,0) + b_{1})$, after which it produces a random number of offsprings or jumps at some other point. Each of the particles evolves independently of others by the same law.

Denote as usual by $E$ the expectation of a random variable.
In the sequel  the main subject of our interest will be the integer moments of the numbers of particles $\eta_{t}(y)$, $\eta_{x,t}$, and $\eta_{x, t}(y) $  at the point $y$, respectively:
\[
M_{\infty, n}(t, y):=E\eta_{t}^{n}(y),\quad M_{n}(t, x):=E\eta_{x, t}^{n},\quad
M_{n}(t, x, y):=E\eta_{x, t}^{n}(y).
\]

Introduce the operator
\[
\mathscr{H}:= \mathscr{A} + \beta\Delta_{0},
\]
where $\Delta_{0}$ is the following operator on $l^{p}(\mathbb{Z}^{d})$ for every $p$, $1 \leq p \leq \infty$:
\[
(\Delta_{0}u)(x):=\delta_{0}(x)u(x)\equiv u(0)\delta_{0}(x), \quad u(\cdot)\in l^{p}(\mathbb{Z}^{d}),~x\in\mathbb{Z}^{d}.
\]
Introduce also the Laplace generating functions for random variables $\eta_{t}(y) $, $\eta_{x,t}$, and $\eta_{x, t}(y) $, defined by the following equations
\[
F_{\infty}(z; t, y) := Ee^{-z\eta_{t}(y)},\quad F_{1}(z; t, x) := Ee^{-z\eta_{x, t}}, \quad
F_{1}(z; t, x, y) := Ee^{-z\eta_{x, t}(y)},
\]
where  $z \geq 0$, the index $\infty$ indicates the generating functions for the whole particle population at $y\in \mathbb{Z}^{d}$, and the index $1$ indicates the generating functions for subpopulations of particles both over  $\mathbb{Z}^{d}$ and at  $y\in \mathbb{Z}^{d}$
generated by the parent particle with the initial location at $x \in \mathbb{Z}^{d}$.

 Now we focus on the derivation of the backward Kolmogorov differential equations for the subpopulations.
\begin{lemma}\label{L1} For each $ 0\leqslant z \leqslant \infty $ the generating functions $F_{1}(z; t, x) $ and $ F_{1}(z; t, x, y) $ are continuously differentiable with respect to $t$ uniformly in $ x,y\in \mathbb{Z}^d$, satisfy the inequalities $ 0\leqslant F_{1}(z; t, x), F_{1}(z; t, x, y) \leqslant 1 $ and the backward Kolmogorov differential equations
\begin{align}\label{E1.1}
\partial_{t}F_{1}(z; t, x) &= (\mathscr{A}F_{1}(z; t, \cdot))(x) + \delta_{0}(x)f(F_{1}(z; t, x)),\\
\label{E1.2}
\partial_{t}F_{1}(z; t, x, y) &= (\mathscr{A}F_{1}(z; t, \cdot, y))(x) + \delta_{0}(x)f(F_{1}(z; t, x, y)),
\end{align}
with the initial conditions $ F_{1}(z; 0, y) = e^{-z} $ and $ F_{1}(z; 0, x, y) = e^{-z\delta_{x}(y)}$, respectively.
\end{lemma}
\begin{proof} The proof of \eqref{E1.1} is by analogy with \cite[Lemma~1.2.1]{YarBRW:r}. We confine ourselves to proving the statement of the lemma related to $F_{1}(z; t, x, y)$.

The inequality $ 0\leqslant F_{1}(z; t, x, y) \leqslant 1 $ results from nonnegativeness of the random variable $\eta_{t}(y)$ and the definition of the generating function.
Taking into account all possible evolutions of the system on the interval $[t, t+h]$ and using the Markov property we obtain:
\begin{align*}
F_{1}(z; t+h, x, y) &= \sum_{x'\neq x}p(h, x, x')E e^{-z \tilde{\eta}_{t}(y)}+ \delta_{0}(x) \sum_{n\neq 1}p_{*}(h, n)Ee^{-z(\tilde{\eta}_{t}^{1}(y) + \cdots + \tilde{\eta}_{t}^{n}(y))}\\
&\quad+ [(1-\delta_{0}(x))p(h, x, x) + \delta_{0}(x)p_{*}(h, 1)]F_{1}(z; t, x, y)\\
&=  \sum_{x'\neq x}a(x, x')hE e^{-z \tilde{\eta}_{t}(y)} + \delta_{0}(x) \sum_{n\neq 1}b_{n}hE e^{-z(\tilde{\eta}_{t}^{1}(y) + \cdots + \tilde{\eta}_{t}^{n}(y))}\\
&\quad+ [1 + a(x, x)h + \delta_{0}(x)b_1 h]F_{1}(z; t, x, y)\\
&\quad + S_1(h) + S_2(h) + S_3(h) + S_4(h)\\
&= F_{1}(z; t, x, y) + [(\mathscr{A}F_{1}(z; t, \cdot, y))(x) + \delta_{0}(x)f(F_{1}(z; t, x, y))]h\\
&\quad+ S_1(h) + S_2(h) + S_3(h) + S_4(h),
\end{align*}
where
\begin{align*}
S_1(h) &= \sum_{x'\neq x}[p(h, x, x') - a(x, x')h ]E e^{-z \tilde{\eta}_{t}(y)}x, \\
S_2(h) &= \delta_{0}(x)\sum_{n\neq 1}[p_{*}(h, n) - b_n h ]E_{\infty}e^{-z(\tilde{\eta}_{t}^{1}(y) + \cdots + \tilde{\eta}_{t}^{n}(y))},\\
S_3(h) &= [p(h, x, x) - 1- a(x, x)h]F_{1}(z; t, x, y),\\
S_4(h) &= \delta_{0}(x) [p_{*}(h, 1) - 1 - b_1 h]F_{1}(z; t, x, y).
\end{align*}
Here $ \tilde{\eta}_{t}(y) $ is a new process, starting at time $ t=h $ at the point $ x' $, which the particle reaches by a jump from the point $ x $, and
$\tilde{\eta}_{t}^{1}(y), \ldots , \tilde{\eta}_{t}^{n}(y) $ are independent processes born in time $h$, produced by $n$ offsprings of the initial particle.

From above, it follows that
\begin{multline}\label{E1.3}
\left| \frac{ F_{1}(z; t+h, x, y) -  F_{1}(z; t, x, y)} {h}  - (\mathscr{A}F_{1}(z; t, \cdot, y))(x) + \delta_{0}(x)f(F_{1}(z; t, x, y)) \right| \\
\leqslant \sum_{x'\neq x} \left|  \frac{p(h, x, x')}{h} - a(x, x') \right|  + \sum_{n\neq 1} \left|  \frac{p_{*}(h, n)}{h} - b_n \right| \\
+ \left| \frac{p(h, x, x)}{h} - \frac{1}{h} - a(x,x) \right|  + \left| \frac{p_{*}(h, 1)}{h} - \frac{1}{h} -a(0,0) - b_1 \right|.
 \end{multline}

The last two summands on the right side of equation~\eqref{E1.3} tend to zero when $h \to 0$ in view of~\eqref{E:phxx} and~\eqref{E:ph1}. Fix an arbitrary $ K > 0 $ and estimate the first summand on the right side of equation~\eqref{E1.3}:
\begin{multline}\label{I1}
\sum_{x'\neq x} \left|  \frac{p(h, x, x')}{h} - a(x, x') \right|
\leqslant \sum_{0 < | x' - x| \leqslant K} \left|  \frac{p(h, x, x')}{h} - a(x, x') \right|\\  + \sum_{| x' - x| > K} a(x, x')  + \sum_{| x' - x| > K} \frac{p(h, x, x')}{h}.
\end{multline}

Rewrite now the last summand on the right side of~\eqref{I1}:
\begin{multline}\label{AE1}
\sum_{| x' - x| > K} \frac{p(h, x, x')}{h} = \frac{1 - p(h, x, x)}{h} - \sum_{0 < | x' - x| \leqslant K}  \frac{p(h, x, x')}{h} \\
= \left( \frac{1 - p(h, x, x)}{h} + a(x, x) \right) - \sum_{0 < | x' - x| \leqslant K} \left(  \frac{p(h, x, x')}{h} - a(x, x') \right)\\
-  a(x,x) -  \sum_{0 < | y - x| \leqslant K} a(x, y)
\end{multline}

Substituting $a(x,x)=-\sum_{y\neq x}a(x,y)$ in~\eqref{AE1}, we finally estimate~\eqref{I1}:

\begin{multline*}
\sum_{x'\neq x} \left|  \frac{p(h, x, x')}{h} - a(x, x') \right| \leqslant 2 \sum_{0 < | x' - x| \leqslant K} \left|  \frac{p(h, x, x')}{h} - a(x, x') \right|  + 2 \sum_{| x' - x| > K} a(x, x')\\
+  \left|  \frac{1-p(h, x, x)}{h} + a(x, x) \right|.
\end{multline*}
Here the last summand tends to zero when $ h \to 0 $ by virtue of the relation~\eqref{E:phxx}; the second sum could be made arbitrary small by choosing sufficiently large  $K$ (since the series $\sum_{x'}a(x, x')$ converges absolutely); with the fixed $K$ the first sum tends to zero when $ h \to 0 $ (since the number of summands is finite and each summand tends to zero when $ h \to 0 $ due to the relation~\eqref{E:phxy}).

Similarly we prove that the second sum in equation~\eqref{E1.3} tends to zero. As a result, it is proved that $ F_{1}(z; t, x, y) $ is differentiable with respect to $t$ (and also continuous) for each $z, x, y$ and satisfies the differential equation~\eqref{E1.2}.

Now the inequalities $ 0\leqslant F_{1}(z; t, x, y) \leqslant 1 $, equation~\eqref{E1.1} and the boundedness of the operators $\mathscr{A}$ and $ \delta_0 $ in the space $l^{\infty}(\mathbb{Z}^d)$ (by the Schur's lemma) lead~\cite{Hartman82} to the uniform boundedness with respect to  $z, x, y $ and continuity of the derivative with respect to $t$ of the function $F_{1}(z; t, x, y) $. And this results in uniform with respect to $z, x, y$ continuous differentiability of the function $ F_{1}(z; t, x, y)$ with respect to $t$.
\end{proof}

\begin{remark}\label{Rem1}\rm
Owing to the uniform with respect to $ y\in \mathbb{Z}^d $ continuous differentiability of the function $ F_{1}(z; t, x) $ with respect to $t$ (for each $ 0\leqslant z \leqslant \infty $), provided by  Lemma~\ref{L1}, equation~\eqref{E1.1} can be treated as the Cauchy problem in the Banach space $l^{\infty}(\mathbb{Z}^d)$ dependent on the parameter $z$:
\begin{equation}\label{E1.4}
 \frac{dF_{1}(z; t, \cdot)}{dt} = \mathscr{A}F_{1}(z; t, \cdot) + \Delta_{0}f(F_{1}(z; t, \cdot))
\end{equation}
with the initial condition $ F_{1}(z; 0, \cdot) = e^{-z} $.

Similarly, since by Lemma~\ref{L1} the function $ F_{x}(z; t, x, y) $ is uniformly (with respect to $ x, y\in \mathbb{Z}^d $) continuously differentiabile in $t$ for each $0\leqslant z \leqslant \infty $, then~\eqref{E1.2} can also be treated as the Cauchy problem in the Banach space $l^{2}(\mathbb{Z}^d)$ dependent on the parameters $z$ and $y$:
\begin{equation}\label{E1.5}
\frac{dF_{1}(z; t, \cdot, y)}{dt} = \mathscr{A}F_{1}(z; t, \cdot, y) +  \Delta_{0}f(F_{1}(z; t, \cdot, y))
\end{equation}
with the initial condition $ F_{1}(z; 0, x, y) = e^{-z\delta_{x}(y)} $.
\end{remark}

Because the right sides of equations \eqref{E1.4} and \eqref{E1.5} are continuous and (under the assumption $\beta = f'(1) < \infty$) satisfy the Lipschitz condition as operators in $l^{\infty}(\mathbb{Z}^d)$ and $l^{2}(\mathbb{Z}^d)$, owing to \eqref{E1.4} and \eqref{E1.5} by the theorem on continuous dependence of the solutions of differential equation on initial conditions and parameters~\cite{Hartman82} the functions $ F_{1}(z; t, \cdot) $ and $ F_{1}(z; t, \cdot, y) $ depend continuously with respect to the norm on the spaces $l^{\infty}(\mathbb{Z}^d)$ and $l^{2}(\mathbb{Z}^d)$ on the parameter $z$ for each $y$ and $ z \geqslant 0$. Besides, by the theorem about differentiability of the solutions of differential equations with respect to a parameter the functions $ F_{1}(z; t, \cdot) $ and $ F_{1}(z; t, \cdot, y) $ are continuously differentiable with respect to $z$ for $ z>0 $ as many times as the function $f(u)$, $ u \in [0, 1] $.\qed

Let us now turn to the study of the first moments. To derive the differential equations for the first moments we will use the generating functions and obtained earlier Lemma~\ref{L1} with Remark~\ref{Rem1}. For it, we express the moments in terms of the Laplace generating functions:
\[
M_{1}(t, x) = - \lim_{z\to 0+}\partial_{z}F_{1}(z; t, x),\quad
M_{1}(t, x, y) = - \lim_{z\to 0+}\partial_{z}F_{1}(z; t, x, y).
\]

\begin{lemma}\label{L2} The moments $M_{1}(t, x)$ and $M_{1}(0, x, y)$ satisfy the following backward Kolmogorov differential equations:
\begin{align}\label{E1.6}
\frac{\partial{M_{1}(t, x)} }{\partial{t}} &= (\mathscr{H}M_{1}(t, \cdot))(x),&&
M_{1}(0, x) \equiv 1,\quad x\in \mathbb{Z}^d,~t\geq 0,\\
\label{E1.7}
\frac{\partial{M_{1}(t, x, y)} }{\partial{t}} &= (\mathscr{H}M_{1}(t, \cdot, y))(x),&&
M_{1}(0, x, y) = \delta_{x}(y),\quad t\geq 0
\end{align}
\end{lemma}

\begin{proof} Let us derive equation \eqref{E1.6}; equation \eqref{E1.7} can be derived similarly. By differentiating the equation
 $$ \partial_{t}F_{1}(z; t, x, y) = (\mathscr{A}F_{1}(z; t, \cdot, y))(x) + \delta_{0}(x)f(F_{1}(z; t, x, y)), $$
 with respect to $z$ (which is possible due to Remark~\ref{Rem1}) we get
$$ - \partial_{t}\partial_{z}F_{1}(z; t, x, y) = -(\mathscr{A}\partial_{z} F_{1}(z; t, \cdot, y))(x) + \delta_{0}(x)\partial_{z}f(F_{1}(z; t, x, y)).
$$
Here, expressing the last term with the help of following formula for the derivative of the superposition $f(F_{1}(z; t, x, y))$ of the functions $f$ and $F_{1}(z; t, x, y)$:
$$
\partial_{z}f(F_{1}(z; t, x, y)) = f'(F_{1}(z; t, x, y))\cdot \partial_{z}F_{1}(z; t, x, y),
$$
we obtain the equation:
$$
- \partial_{t}\partial_{z}F_{1}(z; t, x, y) = -(\mathscr{A}\partial_{z} F_{1}(z; t, \cdot, y))(x) + \delta_{0}(x)f'(F_{1}(z; t, x, y))\cdot \partial_{z}F_{1}(z; t, x, y).
$$
From here, by passing to the limit with respect to $ z\rightarrow 0 $ we get equation~\eqref{E1.6}.
\end{proof}

\begin{remark}\label{Rem2}\rm
Owing to Remark~\ref{Rem1} equations \eqref{E1.6} and \eqref{E1.7} can be interpreted as the Cauchy problems in the Banach spaces $l^{\infty}(\mathbb{Z}^d)$ and $l^{2}(\mathbb{Z}^d)$, and the solutions of these Cauchy problems are unique and belong to the Banach spaces $l^{\infty}(\mathbb{Z}^d)$ and $l^{2}(\mathbb{Z}^d)$.
\end{remark}

\begin{lemma}\label{L3} For $ n\geq 0 $ and $ t\geq 0 $ the higher moments of subpopulations: $\eta_{x,t}$ and $\eta_{x,t}(y)$, at $y\in \mathbb{Z}^{d}$, whose initial position was at point $x$, satisfy the following backward Kolmogorov differential equations:
\begin{equation}\label{E1.8}
\partial_{t}M_{n}(t, x) = (\mathscr{H}M_{n}(t, \cdot))(x) + \delta_0(x)g_{n}(M_{1}(t,x),\ldots, M_{n-1}(t, x)),
\end{equation}
with the initial condition $M_{n}(0, x) \equiv 1$  or
\begin{equation}\label{E1.9}
\partial_{t}M_{n}(t, x, y) = (\mathscr{H}M_{n}(t, \cdot, y))(x) + \delta_0(x)g_{n}(M_{1}(t,x, y),\ldots, M_{n-1}(t, x, y))
\end{equation}
with the initial condition $ M_{n}(0, x, y) = \delta_{x}(y)$, $ y\in \mathbb{Z}^d$,
where
\begin{align*} g_{n}(M_{1}, M_{2},\ldots, M_{n-1}) &=  \sum_{r=1}^{n} \frac{\beta^{(r)}}{r!} \sum_{\stackrel{i_{1},\ldots,i_{r}>0}{i_{1}+\cdots+i_{r}=n}}
\frac{n!}{i_{1}!\cdots i_{r}!} M_{i_{1}}\cdots M_{i_{r}}.
\end{align*}
\end{lemma}

\begin{proof} The proof of the lemma is by analogy with \cite{YarBRW:r}. Again, we confine ourselves to proving only equation~\eqref{E1.9}.

Let us differentiate $n$ times with respect to $z$ the equation
$$
\partial_{t}F_{1}(z; t, x, y) = (\mathscr{A}F_{1}(z; t, \cdot, y))(x) + \delta_{0}(x)f(F_{1}(z; t, x, y)),
$$
(it is permitted owing to Remark~\ref{Rem1}). Then we obtain:
$$
(-1)^n \partial_{t} \partial_{z}^{n}F_{1}(z; t, x, y) = (-1)^{n}(\mathscr{A}\partial_{z}^{n}F_{1}(z; t, \cdot, y))(x) + \delta_{0}(x)\partial_{z}^{n}f(F_{1}(z; t, x, y))
$$

Using the Faa di Bruno's formula~\cite{M:00} we calculate the second summand on the right side of the equation:
\begin{equation}\label{E:df}
\frac{d^{n}f(F(z))}{dz^{n}} = \sum_{r=1}^{n} f^{(r)}(F(z)) \sum_{\stackrel {i_{1}+\cdots+i_{n}=r}{
		i_{1}+2i_{2}+\cdots+ni_{n}=n}}
\frac{n!}{i_{1}!\cdots i_{n}!} \left( \frac{F^{(1)}(z)}{1!} \right)^{i_{1}}\cdots \left( \frac{F^{(n)}(z)}{n!} \right)^{i_{n}}.\end{equation}

Let us prove the supportive equation below in order to simplify the inner sum in this formula to the required in the lemma's statement form:
\begin{equation}\label{E:star} \sum_{\stackrel {i_{1}+\cdots+i_{n}=r}{i_{1}+\cdots+ni_{n}=n}}
\frac{n!}{i_{1}!\cdots i_{n}!}x_{1}^{i_{1}}\cdots x_{n}^{i_{n}} = \sum_{\stackrel {i_{1},\ldots,i_{r}>0}{i_{1}+\cdots+i_{r}=n}} x_{i_{1}}\cdots x_{i_{r}},
\end{equation}
where $ x_{1},\ldots, x_{n} $ are formal variations. Consider the polynomial $ P(t) := (x_{1}t + \cdots +x_{n}t^{n})^{r} $ and calculate its coefficient at $t^{n} $. By the polynomial formula this coefficient coincides with the left side of~\eqref{E:star}. On the other side, by direct gathering the components in the product  $ (x_{1}t + \cdots +x_{n}t^{n})^{r} $ at $ t^{n} $, we get the right side of the formula~\eqref{E:star}.

Thus equation~\eqref{E:df} is proved. Then rewrite the Faa di Bruno's formula:

$$
\frac{d^{n}f(F(z))}{dz^{n}} = \sum_{r=1}^{n} \frac{f^{(r)}(F(z))}{r!} \cdot \sum_{ \stackrel{i_{1},\ldots,i_{r}>0}{i_{1}+\cdots+i_{r}=n}}
\frac{n!}{i_{1}!\cdots i_{n}!} F^{(i_{1})}(z) \cdots  F^{(i_{r})}(z).
$$

Let $ F(z) = F_{1}(z; t, x, y) $. Then
\begin{multline*}
\delta_{0}(x)\partial_{z}^{n}f(F_{1}(z; t, x, y)) =  \sum_{r=1}^{n} \frac{f^{(r)}(F_{1}(z; t, x, y)) }{r!}\\
\times \sum_{ \stackrel{i_{1},\ldots,i_{r}>0}{i_{1}+\cdots+i_{r}=n}}
\frac{n!}{i_{1}!\cdots i_{n}!} \partial_{z}^{i_{1}} F_{1}(z; t,x,y) \cdots  \partial_{z}^{i_{r}}F_{1}(z, t, x, y).
\end{multline*}
Using the expression of the moments in terms of generating functions, by the passage to the limit with respect to $z\rightarrow 0$ we get the required equation~\eqref{E1.9}.
\end{proof}

\begin{lemma}\label{L4}
For $ n> 0$ and $ t\geq 0 $ the higher moments of subpopulations: $\eta_{x,t}$ and $\eta_{x,t}(y)$, at $y\in \mathbb{Z}^{d}$, whose initial position was at point $x$, satisfy the following integral equations:
\begin{equation}\label{E1.10}
M_{n}(t, x) = M_{1}(t, x) + \delta_{0}(x)\int_{0}^{t}M_{1}(t-q, x, 0)g_{n}(M_{1}(q,0),\ldots, M_{n-1}(q, 0))dq,
\end{equation}
\begin{multline}\label{E1.11}
M_{n}(t, x, y) = M_{1}(t, x, y)\\ + \delta_{0}(x)\int_{0}^{t}M_{1}(t-q, x, 0)g_{n}(M_{1}(q,0, y),\ldots, M_{n-1}(q, 0, y))dq,
\end{multline}
where
\begin{align*}
g_{n}(M_{1}, M_{2},\ldots, M_{n-1}) &=  \sum_{r=1}^{n} \frac{\beta^{(r)}}{r!} \cdot \sum_{ \stackrel{i_{1},\cdots,i_{r}>0}{i_{1}+\cdots+i_{r}=n}}
\frac{n!}{i_{1}!\cdots i_{n}!} M_{i_{1}} \cdots M_{i_{r}}.
\end{align*}
\end{lemma}

\begin{proof} The proof of the lemma is by analogy with \cite{YarBRW:r}. Using  variation of parameters in~\eqref{E1.8} and~\eqref{E1.9}, we obtain the required equations.
\end{proof}

In conclusion of this section let us recall the model of BRW defined in \cite{YarBRW:r}.
Let at time $t=0$ there be only one particle at $x\in \mathbb{Z}^{d}$, from which the evolution of BRW starts. Denote by $\mu_{x,t}$ the total number of particles on the lattice at time $t$ (the population of particles on $\mathbb{Z}^{d}$), and by $\mu_{x,t}(y)$ the number of particles at $\mu_{x,t}(y)$. They satisfy the initial conditions $ \mu_{x,0}=1 $ and $\mu_{x,0}(y)=\delta_{x}(y)$, respectively. Following~\cite{YarBRW:r}  let us consider the Laplace generating functions:
\[
F (z; t, x) := Ee^{-z\mu_{x, t}}, \quad
F (z; t, x, y) := Ee^{-z\mu_{x, t}(y)}.
\]

\begin{theorem}\label{L:osn} For each $ 0\leqslant z \leqslant \infty $ the generating functions $F_{1}(z; t, x) $ and $ F (z; t, x) $ are continuously differentiable with respect to $t$ uniformly in $ x\in \mathbb{Z}^d$, satisfy the inequalities $ 0\leqslant F_{1}(z; t, x), F (z; t, x) \leqslant 1 $ and the backward Kolmogorov differential equation~\eqref{E1.1}; the generating functions $F_{1}(z; t, x,y) $ and $ F (z; t, x,y) $ are continuously differentiable with respect to $t$ uniformly in $ x,y\in \mathbb{Z}^d$, satisfy the inequalities $ 0\leqslant F_{1}(z; t,x,y), F (z; t, x,y) \leqslant 1 $ and the backward Kolmogorov differential equation~\eqref{E1.2}.
\end{theorem}
\begin{proof} By Lemma~\ref{L1} we obtain the assertion of the theorem for $F_{1}(z; t, x)$ and $F_{1}(z; t,x, y)$ and by \cite[Lemma~1.2.1]{YarBRW:r} for $F(z; t, x)$ and $F(z; t,x, y)$.
\end{proof}
Theorem~\ref{L:osn} implies that for each $n\geq 1$ the bahaviour of $m_{n}(t,x):=E\mu_{x,t}^{n}$ in the model described in \cite{YarBRW:r} coincides with the behaviour of $M_{n}(t,x)$  in the model under consideration, and for each $n\geq 1$ the behaviour of $m_{n}(t,x,y):=E\mu_{x,t}^{n}(y)$ in the model described in \cite{YarBRW:r} coincides with the behaviour of  $M_{n}(t,x,y):=E\eta_{x,t}^{n}(y)$. This demonstrates that the results obtained for BRWs with one initial particle may be applied for the study of BRWs with infinitely many initial particles.
Owing to Theorem~\ref{L:osn} we can rewrite the limit theorem, proved in \cite{YarBRW:r} for $M_{n}(t, x, y)$ and $ M_{n}(t, x)$. Recall~\cite{YarBRW:r} that the intensity value $\beta_{c}$ is called \emph{critical} if for $\beta<\beta_{c}$ all the moments tend to zero or bounded on infinity, while for $\beta>\beta_{c}$ all the moments grow exponentially. As is known~\cite{YarBRW:r} in this case the growth rates of moments is closely related to the largest positive eigenvalue $\lambda_{0}$ of the operator $\mathscr{H}=\mathscr{A}+ \beta\Delta_{0}$.

\begin{theorem} \label{T3}
Let $\sum_{x\in \mathbb{Z}^{d}} a(0,x)|x|^{2}<\infty$ then
for all $ n \in \mathbb{N}$ and $ t \rightarrow \infty$
\[
M_{n}(t, x, y)  \sim C_{n}(x, y)u_{n}(t), \quad M_{n}(t, x) \sim C_{n}(x)v_{n}(t),
\]
 where $ C_{n}(y, x) > 0,  C_{n}(x) > 0$, and the function $ v_{n}(t) $ is expressed by:

\noindent a) for $ \beta > \beta_{c} $,
$$ u_{n}(t) = e^{n\lambda_0 t}, \quad  v_{n}(t) = e^{n\lambda_0 t};$$
\noindent b) for $ \beta = \beta_{c} $:
\begin{align*}
d&=1: &u_{n}(t)&= t^{(n-1)/2}(\ln t)^{n-1}, &v_{n}(t)&= t^{(n-1)/2};\\
d&=2: &u_{n}(t)&= t^{-1}, &v_{n}(t)&= (\ln t)^{n-1};\\	
d&=3: &u_{n}(t)&= t^{-1/2}(\ln t)^{n-1}, &v_{n}(t)&= t^{n-1/2};\\	
d&=4:  &u_{n}(t)&= t^{n-1}(\ln t)^{1-2n}, &v_{n}(t)&= t^{2n-1}(\ln t)^{1-2n};\\	
d&\geq 5: &u_{n}(t)&= t^{2n-1}, &v_{n}(t)&= t^{2n-1};
\end{align*}
\noindent c) for $ \beta < \beta_{c} $:
\begin{align*}
d&=1:  &u_{n}(t) &= t^{-3/2}, &v_{n}(t) &= t^{-1/2}; \\
d&=2:  &u_{n}(t) &= (t \ln^{2} t)^{-1}, &v_{n}(t)&= (\ln t)^{-1};\\
d&\geq 3:  &u_{n}(t) &= t^{-d/2}, &v_{n}(t)&\equiv 1 .
\end{align*}
\end{theorem}

\section{The duality of BRW models}\label{BRWIF}

The goal of this section is to prove the following relation
\begin{equation}\label{E:Meqm}
M_{\infty, 1}(t, y) \equiv m_{1}(t, y),
\end{equation}
which can be treated as a kind of duality between the of BRW model with a single initial particle and the BRW model  with infinitely many particles.

\begin{theorem}\label{M:osn}
The function $M_{\infty,1}$ satisfies the Cauchy problem
\begin{equation}\label{M1.1}
\partial_{t}M_{\infty,1}(y) = (\mathscr{A}^{*}M_{\infty,1}(t,\cdot))(y) + \delta_{0}(y)M_{\infty,1}(y), \quad M_{\infty,1}(t,\cdot)=1,
\end{equation}
where
\[
(\mathscr{A}^{*}u)(x):=\sum_{x'}a(x',x)u(x'), \quad u(\cdot)\in l^{p}(\mathbb{Z}^{d}),~x\in\mathbb{Z}^{d}.
\]
\end{theorem}\label{T:osn}
\begin{remark}\label{R:opA}\rm
Let us note that throughout the paper  $a(x,y)\equiv a(y-x)$, and then $\mathscr{A}^{*}=\mathscr{A}$.
\end{remark}
\begin{proof}
Let us estimate $M_{\infty,1}(t+h,y)$ using the formula of total
probability.  The variation of the number of particles at the point $y$ over the
time interval $[t,t+h]$ up to terms of order $o(h)$ is defined by one
of the following possibilities:  (a) particle jump from some point $y'$,
$y'\neq y$, to $y$, (b) particle jump from the point $y$ to another point
and, finally, (c) branching of particles at the point $y$, subject to the
condition $y=0$.  Hence, taking into account the equality $\sum_{n\neq 1}b_{n}=-b_{1}$, we get
\begin{multline*}
M_{\infty,1}(t+h,y)-M_{\infty,1}(t,y)=E[\eta_{t+h}(y)-\eta_{t}(y)]=
\\=E \Bigg[\sum_{y'\ne y} \eta_{t}(y')a(y',y)
h+\eta_{t}(y)a(y,y)h+
\delta_{0}(y)\sum_{n\ne1}\eta_{t}(y)b_{n}(n-1) h+o(h)\Bigg]=
\\=\sum_{y'} M_{\infty,1}(t,y')a(y',y) h+
\delta_{0}(y)M_{\infty,1}(t,y)f'(1)h+o(h)=
\\=\left[\mathscr{A}^{*}M_{\infty,1}(t,y)+\delta_{0}(y)\beta M_{\infty,1}(t,y)
\right]h+o(h).
\end{multline*}
and passing to the limit for $h\to0$ we obtain \eqref{M1.1}.  Since the
legitimacy of the corresponding passage to the limit is explaied along the
same lines as the corresponding place in the proof of Lemma~\ref{L1},
it is omitted here.
\end{proof}

To prove identity~\eqref{E:Meqm} let us introduce notations
\[
u(t)=M_{\infty, 1}(t,\cdot),\quad v(t)=m_{1}(t,\cdot).
\]
Then both functions, $u(t)$ and $v(t)$, treated as functions in $l^{\infty}(\mathbb{Z}^{d})$ satisfy the same initial condition $u(0)=v(0)=1$ or, what is the same, $u(0,\cdot)=v(0,\cdot)\equiv1$.

According to~\cite[Thm.~1.3.1]{YarBRW:r} the function $v(t)$ is a solution of the following linear differential equation in $l^{\infty}(\mathbb{Z}^{d})$:
\begin{equation}\label{E188}
\frac{dx}{dt}=\mathscr{A}x +\Delta_{0}x=\mathscr{H}x,
\end{equation}
whereas according to Theorem~\ref{M:osn} the function $u(t)$ is a solution of the following linear differential equation in $l^{\infty}(\mathbb{Z}^{d})$:
\[
\frac{dx}{dt}=\mathscr{A}^{*}x +\Delta_{0}x=\mathscr{H}^{*}x.
\]

Since by assumption of this paper the linear operator $\mathscr{A}$ is symmetric, that is $\mathscr{A}=\mathscr{A}^{*}$, then the both functions, $u(t)$ and $v(t)$, satisfy the same linear differential equation~\eqref{E188} with bounded linear operator in the right-hand side, and the same initial conditions. Due to unique dependence of solutions of the differential equations on the initial condition, in this case $v(t)\equiv u(t)$ for all $t\ge0$ or, what is the same,
\[
M_{\infty, 1}(t,\cdot) \equiv u(t,\cdot)\equiv v(t,\cdot)\equiv m_{1}(t,\cdot),
\]
which finalize the proof of~\eqref{E:Meqm}.

\begin{remark}\label{RforwardKolmEq}\rm One of possible ways to prove the identity~\eqref{E:Meqm} is to obtain the forward Kolmogorov equation for $F_{\infty}(z;t,y)$, and then by standard arguments to derive the equation for the first moment. However, this way is technically quite cumbersome, so  we preferred to give in Theorem~\ref{M:osn} a direct proof of the forward Kolmogorov equation  for $M_{\infty, 1}(t, y)$.

Nevertheless, for the sake of completeness, we would like to present here without proofs the forward Kolmogorov equations for the generating functions $F_{\infty }(z; t, y) $ and  $F_{1}(z; t, x, y)$:
\begin{align*}
\nonumber\partial_{t}F_{ \infty }(z; t, y) &= (e^{-z}-1)\sum_{y'\neq y}a(y', y)E\left( e^{-z \eta_{t}(y)}\eta_{t}(y')\right)\\
\nonumber&\quad + (e^{-z}-1)a(y, y) E \left(e^{-z(\eta_{t}(y)-1)} \eta_{t}(y)\right) \\
&\quad + \delta_{0}(y)f(e^{-z})E\left( e^{-z(\eta_{t}(y)-1)}\eta_{t}(y)\right),\\
\nonumber\partial_{t}F_{1}(z; t, x, y) &= (e^{-z}-1)\sum_{y'\neq y}a(y', y)E\left( e^{-z \eta_{x,t}(y)}\eta_{x,t}(y')\right)\\
\nonumber&\quad + (e^{-z}-1)a(y, y) E \left(e^{-z(\eta_{x,t}(y)-1)} \eta_{x,t}(y)\right) \\
&\quad + \delta_{0}(y)f(e^{-z})E\left( e^{-z(\eta_{x,t}(y)-1)}\eta_{x,t}(y)\right)
\end{align*}
where $ F_{\infty }(z; 0, y) = e^{-z} $ and $ F_{1}(z; 0, x, y) = e^{-z\delta_{x}(y)}$, respectively.
\end{remark}

\section{The vaccination process}\label{BRWV}

We would like to suggest one of the possible ways to introduce the vaccination process into the defined models of a viral evolution. We describe in terms of branching random walk the viral behavior in the host's population, where individuals are more resistant to epidemic, due to some preventative medical actions.

Usually in mathematical models the vaccination process of SIR and SEIR viruses is described in the terms of the number of humans in the population, who can be infected by the virus, see, e.g.,~\cite{GMS-B:10,BGMS-B:10,RT:15}. Then the vaccination increases the part of carrier's population, who are immune to the disease \cite{RT:15}.  As distinct from  \cite{GMS-B:10}, in this work the process is described at the level of one particle: the vaccine would be presented by reduction of the viral activity.

Return to the model of branching random walk, defined in \cite{YarBRW:r}. At first, for simplicity we consider the infection, where  the viral particle produces no more than two offsprings in time. This model  is rather universal, as due to  the continuity of the process, it is possible to make the temporal interval small enough, such that only one event happens during this time interval.
Thus, the vaccine would be presented by the reduction of the viral replication possibility. As in terms of this work the coefficients $b_n$,  $n>1$, correspond to the virus capacity, we suggest considering a function $\gamma: \mathbb R_+  \to  \mathbb R_+ $ that reduces the values of $b_i$, $i>1$.

In our terms it means that
\[
\gamma (b_0) = b_0, \qquad \gamma(b_n) \leq b_n \quad n>1.
\]
The value of $b_1$ can be found from the expression for coefficients of probability generating function $\sum\limits_{n=0}^{\infty} b_n = 0$.

In this case the most natural way to define $\gamma$ is to multiply $b_2$ by a constant $\alpha$, where $\quad \alpha \in (0,1)$:
$$
\gamma(b_0) = b_0, \qquad \gamma(b_2) = \alpha b_2, \qquad \gamma(b_1) = 1 - b_0 - \alpha b_2.
$$

Despite the universality of this simplest model, it is interesting to generalize the vaccination process expanding it on the case with a finite number of offsprings at a time. We suggest defining the vaccination process in the following way by changing the branching coefficients:
$$\gamma(b_0) = b_0, \qquad \gamma(b_n) = \alpha^{n-1} b_n, \quad  n \geq 2.$$

This suggestion could be explained from the biological point of view by negative influence of the vaccination on the viral productivity. Due to the nature of the vaccination process as preventative actions, we assume, that the intensity of dying of a viral particle stays stable. But the possibility to produce more than one descendent becomes less than without vaccine. So if the host's organism is vaccinated, that is, immune for a disease, the probability for viral particles to produce $n$ descendants tends to zero with increasing of~$n$.

Define new coefficients  $\tilde{b}_n = \gamma(b_n)$.
As earlier, they satisfy following conditions:
$$
\tilde{b}_n > 0, \quad n \neq 1, \qquad \tilde{b}_1<0, \qquad \sum _{n=0}^{\infty} \tilde{b}_n = 0.
$$

In this case the branching process change its transition probabilities \eqref{pn}. At time $t$ the particles out of the source for a short time $h$ act in the same way as in the models without vaccination, while the particles at the source may jump to the point $y\neq 0$ with probability $\tilde{p}(h, x, y) = a(x, y)h + o(h)$, equal to \eqref{E1}, or  produce $n\neq 1$ offsprings with probability $\tilde{p}_*(h, n) = \tilde{b}_nh + o(h) = \alpha^{n-1} b_nh + o(h)$, unlike \eqref{pn}, or die (case with $n=0$) with probability $\tilde{b}_0h + o(h) = b_0h + o(h) $, or remain unchanged with probability
\begin{equation}\label{newph1}
\tilde{p}_*(h, 1)=1 -\sum _{y\in\mathbb{Z}^d, x\neq y} a(x, y)h -\sum_{y\in\mathbb{Z}^d,x\neq y} o(h) - \sum _i \tilde{b}_{i}h - \sum_{n\neq 1}o(h, n).
\end{equation}

The branching process with vaccination is defined by the generating function
$\tilde{f}(u)$, that can be expressed in terms of $f(u)$ defined by~\eqref{pf}, in the following way:
$$
\quad \tilde{f}(u) =  \frac{ f(0)(1-u)(\alpha - 1) + f(\alpha u) - u f(\alpha)}{\alpha}.
$$

Taking into account the following descriptions of $b_n$:
$$
\tilde{b}_0 = b_0, \quad \tilde{b}_n = \alpha^{n-1} b_n, \quad \tilde{b}_1 = -\sum _{n \neq 1} \tilde{b}_n = -\sum _{n \neq 1} \alpha^{n-1} b_n - b_0,
$$
we obtain that the function $\tilde{f}(u)$ is of the form:

\begin{align*}
\tilde{f}(u) &= \sum _{n = 0} \tilde{b}_n u^n = \tilde{b}_0 +  \tilde{b}_1 u + \sum _{n=2} \tilde{b}_n u^n =\\
&= b_0 + ( - \sum _{n=2} \alpha^{n-1} b_n u^n - b_0) u + \sum _{n=2} \alpha^{n-1} b_n u^n =\\
& = b_0 + \alpha^{-1} ( - \sum _{n=0} \alpha^{n} b_n u^n + b_0 + b_1 \alpha) u - b_0 u +  \alpha^{-1}(\sum _{n=0} \alpha^{n} b_n u^n - b_0 - b_1 \alpha u) =\\
& = b_o +  \alpha^{-1}(-f(\alpha) + b_0 + b_1 \alpha)u -b_0u + \alpha^{-1}(f(\alpha u) - b_0 - b_1 \alpha u) = \\
& = b_0 -  \alpha^{-1} f(\alpha) u + b_0  \alpha^{-1} u + b_1u - b_0 u +  \alpha^{-1} f(\alpha u) -  \alpha^{-1} b_0 - b_1 u =\\
& = \frac{ f(0)(1-u)(\alpha - 1) + f(\alpha u) - u f(\alpha)}{\alpha}.
\end{align*}
With this expression for the generating function the transition probabilities $\tilde{p}(h, x, y)$  and $\tilde{p}_*(h, n)$ can be uniquely found. Eventually, the expressions for the Laplace generating functions in terms of  $\tilde{p}(h, x, y)$ and $\tilde{p}_*(h, n)$ are the same as in the models without vaccination and equation \eqref{E1.2} is valid with the new probability generating function $\tilde f(u)$.

\begin{lemma}\label{L4.1} For each $ 0\leqslant z \leqslant \infty $ the generating function $\tilde{F}_{1}(z; t, x, y) $ is continuously differentiable with respect to $t$ uniformly in $x,y\in \mathbb{Z}^d$, satisfies the inequalities $ 0\leqslant  \tilde{F}_{1}(z; t, x, y) \leqslant 1 $ and the differential equations
\begin{align}\label{E4.1}
\partial_{t}\tilde{F}_{1}(z; t, x, y) &= (\mathscr{A}\tilde{F}_{1}(z; t, \cdot, y))(x) + \delta_{0}(x)\tilde{f}(\tilde{F}_{1}(z; t, x, y)),
\end{align}
with the initial condition $ \tilde{F}_{1}(z; 0, x, y) = e^{-z\delta_{x}(y)}$.
\end{lemma}

The proof of this lemma verbatim repeats the proof of Lemma~\ref{L1} and so is omitted.
According to Lemma~\ref{L4.1} the function $\tilde{F}_{1}(z; t, x, y)$ in the described model with vaccination satisfies equation~\eqref{E4.1}  which is similar to that used in~\cite{YarBRW:r} but with another probability generation function $\tilde{f}$. The next step of research is to derive the equations for the first and the higher moments of BRWs and to find how the limit theorems will change depending on the value of the coefficient $\alpha$.

The presented  vaccination process changes only the properties of the source, while the transition mechanisms of particles remain intact. Consequently it is suitable to the both models of BRWs, mentioned in the work.

The only changes at the branching source in the model with infinite number of initial particles could be explained by the fact, that the vaccination of carriers is applied only in the focus of outbreak, what can possibly fits with real data. Under the assumption that the proposed scheme of vaccination action  corresponds to the real viral evolution, the prediction of epidemic extension, dependent on the vaccine success presented by $\alpha$,  becomes possible.

\textbf{Acknowledgements.} The research was supported by the Russian Foundation for
Basic Research, project no. 17-01-00468.

\end{document}